\documentclass[graybox]{svmult}

\usepackage{type1cm}        
\usepackage{makeidx}         
\usepackage{graphicx}        
\usepackage{multicol}        
\usepackage[bottom]{footmisc}

\usepackage{newtxtext}       %
\usepackage{newtxmath}       

\makeindex             

\newcommand{\bA}{\ensuremath{\mathbf{A}}}
\newcommand{\bx}{\ensuremath{\mathbf{x}}}
\newcommand{\by}{\ensuremath{\mathbf{y}}}
\newcommand{\bb}{\ensuremath{\mathbf{b}}}
\newcommand{\0}{\ensuremath{^{0}}}
\newcommand{\s}{\ensuremath{^{s}}}
\newcommand{\bN}{\ensuremath{\mathbf{N}}}
\newcommand{\bM}{\ensuremath{\mathbf{M}}}
\newcommand{\bR}{\ensuremath{\mathbf{R}}}
\newcommand{\bH}{\ensuremath{\mathbf{H}}}
\newcommand{\bB}{\ensuremath{\mathbf{B}}}
\newcommand{\bD}{\ensuremath{\mathbf{D}}}
\newcommand{\bV}{\ensuremath{\mathbf{V}}}

\newcommand{\bL}{\ensuremath{\boldsymbol{\Lambda}}}

\newcommand{\alert}[1]{{\color{black} #1}}

\graphicspath{{Figures.d/}}

\begin{document}

\title*{Toward a new fully algebraic preconditioner for symmetric positive definite problems}
\author{Nicole Spillane}
\institute{Nicole Spillane \at CNRS, CMAP, Ecole Polytechnique, Institut Polytechnique de Paris, 91128 Palaiseau Cedex, France, \email{nicole.spillane@cmap.polytechnique.fr}}
%
%
\maketitle

\abstract*{}

\abstract{A new domain decomposition preconditioner is introduced for efficiently solving linear systems $\bA \bx = \bb$ with a symmetric positive definite matrix $\bA$. The particularity of the new preconditioner is that it is not necessary to have access to the so-called \textit{Neumann} matrices (\textit{i.e.}: the matrices that result from assembling the variational problem underlying $\bA$ restricted to each subdomain). All the components in the preconditioner can be computed with the knowledge only of $\bA$ (and this is the meaning given here to the word \textit{algebraic}). The new preconditioner relies on the GenEO coarse space for a matrix that is a low-rank modification of $\bA$ and on the Woodbury matrix identity. The idea underlying the new preconditioner is introduced here for the first time with a first version of the preconditioner. Some numerical illustrations are presented. \alert{A more extensive presentation including some improved variants of the new preconditioner can be found in \cite{AlgebraicGenEO}.}} 
 
\section{Introduction}
\label{sec:intro}
We set out to solve the linear system 
$\bA\bx_* = \bb$, 
for a given symmetric positive definite (spd) matrix \alert{$\bA \in \mathbb R^{n\times n}$. } 
There exist a variety of two-level methods for which \alert{fast} convergence is guaranteed without making assumptions on the number of subdomains, their shape, or the distribution of the coefficients in the underlying PDE (see \textit{e.g} \cite{mandel2007adaptive,efendiev2012robust,spillane2013abstract,SPILLANE:2013:FETI_GenEO_IJNME,haferssas2017additive,klawonn2016adaptive,pechstein2017unified,gander2017shem,zampini2016pcbddc,yu2020additive,Dolean:ATL:2011}). These methods have in common to select vectors for the coarse space by computing low- or high-frequency eigenvectors of \alert{well-chosen generalized eigenvalue problems (of the form $\bM_A \by = \lambda \bM_B \by$) posed in the subdomains. To the best of the author's knowledge, none of these methods can be applied if the so-called local $\textit{Neumann}$ matrices are not known. Specifically,  the definition of either $\bM_A$ or $\bM_B$ is based on a family of symmetric positive semi-definite (spsd) matrices $\bN\s$ that satisfy
\begin{equation}
\label{eq:splitting}
\exists C>0, \text{ such that } \sum_{s=1}^N \bx^\top  {\bR\s}^\top \bN\s \bR\s \bx \leq C \, \bx^\top \bA \bx; \, \forall \, \bx \in \bR^n, 
\end{equation}
where it has been assumed that there are $N$ subdomains with restriction operators $\bR\s$.}
The Neumann matrices are a natural choice for $\bN \s$ and the above estimate then holds with constant $C$ equal to the maximal multiplicity of a mesh element. This limitation is very well known (and stated clearly in \textit{e.g.:}, \cite{agullo2019robust,aldaas:hal-01963067}). 

In this work, it is proposed to relax the assumptions on \alert{the matrices $\bN \s$ in \eqref{eq:splitting} by allowing them to be symmetric (but not necessarily positive semi-definite)}. Such matrices $\bN\s$, \alert{then denoted $\bB\s$,} can always be defined \alert{algebraically}. Special treatment must be applied to the non-positive part of $\bB\s$ and this will be reflected in the cost of setting up and applying the preconditioner. In Section~\ref{sec:Theory}, the new preconditioner is defined and the result on the condition number is given. In Section~\ref{sec:Numerical}, some preliminary numerical illustrations are provided. Finally, Section~\ref{sec:Conclusion} offers up some conclusive remarks about the new preconditioner, as well as some of its current limitations that are addressed in the full length article \cite{AlgebraicGenEO}.  

\section{Definition of the new preconditioner and theory}
\label{sec:Theory}

\alert{
This section introduces the new preconditioner $\bH(\tau)$ and proves the resulting bound for the condition number of $\bH(\tau) \bA$. The methodology is as follows. In Subsection~\ref{sub:alg-DD}, some elements of the abstract Schwarz setting are defined in their algebraic form. Then, in Subsection~\ref{subs:def-of-A+}, a new matrix $\bA_+$ is introduced for which an algebraic splitting into spsd matrices is available by construction (\textit{i.e.}, \eqref{eq:splitting} is satisfied). The availability of this splitting makes it possible to apply the abstract GenEO theory \cite{spillane:hal-03186276} to choose a coarse space. Hence, in Subsection~\ref{subs:def-H+}, a two-level preconditioner $\bH_+(\tau)$, with a GenEO coarse space parametrized by a threshold $\tau$, is defined for $\bA_+$. The spectral bound for $\bH_+(\tau) \bA_+$ is given. Finally in Subsection~\ref{subs:def-of-H}, the Woodbury matrix identity \cite{woodbury1950inverting} is applied to find a formula for $\bA ^{-1} - \bA_+^{-1}$ and this (provably low-rank) term is added to $\bH_+(\tau)$ in order to form the new preconditioner $\bH(\tau)$ for $\bA$.  A spectral bound for $\bH(\tau)\bA$ follows.}

\alert{
\subsection{Algebraic Domain Decomposition}
\label{sub:alg-DD}
Let $\Omega = \llbracket 1, n \rrbracket$ be the set of all indices in $\mathbb R^n$. In all that follows, it is assumed that $\Omega$ has been partitioned into a family of subdomains $\left( \Omega\s \right)_{s=1,\dots,N}$ and that the partition has minimal overlap in the sense given by Definition~\ref{def:partition}. The usual restriction operators are also defined.
 
\begin{definition}
\label{def:partition}
A set $\left( \Omega\s \right)_{s=1,\dots,N}$ of $N\in \mathbb N$ subsets of $\Omega = \llbracket 1, n \rrbracket$ is called a partition of $\Omega$ if $\Omega = \bigcup_{s=1}^N \Omega\s$. Each $\Omega\s$ is called a subdomain.  
The partition is said to have at least minimal overlap if: for any pair of indices $(i,j) \in \llbracket 1, n \rrbracket^2$, denoting by $A_{ij}$ the coefficient of $\bA$ at the $i$-th line and $j$-th column, 
\[
A_{ij} \neq 0 \Rightarrow \left( \exists \, s \in \llbracket 1, N \rrbracket \text{ such that } \{i,j\} \subset \Omega\s \right). 
\]
Moreover, for each $s \in \llbracket 1,N \rrbracket$, let $n \s$ be the cardinality of $\Omega\s$. Finally, let the restriction matrix $\bR\s \in \mathbb R^{n\s \times n}$ be zero everywhere except for the block formed by the columns in $\Omega\s$ which is the $n\s\times n\s$ identity matrix. 
\end{definition}
}

\subsection{Definition of $\bA_+$ and related operators}
\label{subs:def-of-A+}

\alert{The starting point for the algebraic preconditioner is to relax condition \eqref{eq:splitting} by allowing symmetric, but possibly indefinite, matrices in the splitting of $\bA$. 

\begin{definition}
Let $\bB \in \mathbb R^{n \times n}$ be the matrix whose $(i,j)$-th entry is 
\[
B_{ij} := \left\{ \begin{array}{cl}
  \frac{A_{ij}}{\# \{s; \{i,j\} \subset \Omega\s\}} &\text{ if } A_{ij} \neq 0,
\\
0  &\text{ otherwise}.
\end{array} \right.
\] 
Then, for each $s=1, \dots, N$, let 
$\bB\s := \bR\s \bB {\bR\s}^\top \quad (\in \mathbb R^{n\s \times n\s})$.  
\end{definition}

\begin{theorem}
Thanks to the minimal overlap assumption, the symmetric matrices $\bB\s$ are well-defined and satisfy 
$\bA = \sum_{s=1}^N {\bR\s}^\top \bB\s \bR\s$.
\end{theorem}

The proof is given in \cite{AlgebraicGenEO}[Theorem 3.2]. In particular, \eqref{eq:splitting} holds with $\bN\s = \bB\s$ and $C=1$. Next, each $\bB\s$ is split into a spsd and a symmetric negative semi-definite part. 
}
\alert{
\begin{definition}
Let $s\in \llbracket 1, N \rrbracket$. Since $\bB\s$ is symmetric, there exist a diagonal matrix $\bL\s$ and an orthogonal matrix $\bV\s$ such that $\bB\s = \bV\s \bL\s {\bV\s}^\top$. It can further be assumed that the diagonal entries of $\bL\s$ (which are the eigenvalues of $\bB\s$) are sorted in non-decreasing order and that 
\[
\bL\s = \begin{pmatrix} \bL\s_- & \mathbf{0} \\ \mathbf{0} & \bL\s_+ \end{pmatrix}, \quad \bV\s = \left[\bV\s_- | \bV\s_+ \right], \quad \bL\s_+ \text { is spd}, \quad -\bL\s_- \text{ is spsd} . 
\]  
Finally, let 
\[
\bA\s_+ := \bV\s_+ \bL\s_+ {\bV\s_+}^\top  \text{ and } \bA\s_- := - \bV\s_- \bL\s_- {\bV\s_-}^\top . 
\]
\end{definition}

With words, the positive (respectively, non-positive) eigenvalues of $\bB\s$ are on the diagonal of $\bL\s_+$ (respectively, $\bL\s_-$) and the corresponding eigenvectors are in the columns of $ \bV\s_+$ (respectively, $\bV\s_-$). It is also clear that
\[
\bB\s = \bA\s_+ - \bA\s_-, \quad \bA\s_+ \text{ is spsd, and } \bA\s_- \text{ is spsd.}   
\]
In the next definition, these new local matrices are assembled into global matrices and in particular the all important matrix $\bA_+$ is defined. 

\begin{definition}
Let $\bA_+$ and $\bA_-$ be the two matrices in $\mathbb R^{n\times n}$ defined by 
\[
\bA_+ := \sum_{s=1}^N {\bR\s}^\top \bA\s_+ \bR\s, \text{ and } \bA_- := \sum_{s=1}^N {\bR\s}^\top \bA\s_- \bR\s.  
\]
\end{definition}
}
\alert{
It is clear that $\bA = (\bA_+ - \bA_-)$ and $\bA_-$ is spsd. As a result, $\bA_+$  is spd .
}
\subsection{Two-level preconditioner for $\bA_+$ with a GenEO coarse space}
\label{subs:def-H+}

Following~\cite{spillane:hal-03186276}, there are many possible choices for a two-level preconditioner for $\bA_+$ with a GenEO coarse space. This is not the novelty here so only one is given with no further comment on other possibilities. 
\alert{
\begin{theorem}
Let $\tau > 1$ be a threshold. Let $\bH_+(\tau)$ be defined by
\[
\bH_+(\tau) := \sum_{s=1}^N {\bR\s}^\top (\bR\s \bA_+   {\bR\s}^\top )\,^{-1} \bR\s +\bR\0(\tau)^\top (\bR\0(\tau) \bA_+ \bR\0(\tau)^\top)^{-1} \bR\0(\tau),
\]
where the lines of $\bR\0(\tau)$ form a basis for the GenEO coarse space $V\0(\tau)$. The coarse space is in turn defined according to \cite{spillane:hal-03186276}[Definition 5] by 
\[
V\0(\tau) := \sum_{s=1}^N \operatorname{span} \left\{{\bR\s}^\top \by\s; (\lambda\s, \by\s) \in \mathbb R^+ \times \mathbb R^{n\s} \text{ solution of \eqref{eq:gevp} and }  \lambda\s < \tau\alert{^{-1}}\right\}.
\]
where the generalized eigenvalue problem is
\begin{equation}
\label{eq:gevp}
(\bD\s)^{-1} \bA_+\s (\bD\s)^{-1} \by\s = \lambda\s \bR\s \bA_+  {\bR\s}^\top \by\s; \text{ for } \bD \s :=  \bR \s \left(\sum_{t=1}^N {\bR^t}^\top  \bR^t \right)^{-1} {\bR \s}^\top. 
\end{equation}
If $\tau>1$ and $\mathcal N_+$ is the minimal number of colors that are needed to color each subdomain in such a way that two subdomains with the same color are $\bA_+$-orthogonal, then the eigenvalues of the preconditioned operator satisfy
\begin{equation}
\label{eq:lambdaH+A+}
\lambda(\bH_+(\tau) \bA_+ ) \in \left[\left((1 + 2 \mathcal N_+)\tau  \right)^{-1}  , \mathcal N_+ + 1\right].
\end{equation}
\end{theorem}
\begin{proof}
This is the result in \cite{spillane:hal-03186276}[Remark 3,Corollary 4,Assumption 6].
\end{proof}
}
\subsection{New preconditioner for $\bA$}
\label{subs:def-of-H}

\alert{
\begin{definition}
\label{def:V-L-}
Let $n_- = \operatorname{rank}(\bA_-)$. Let  $\bL_- \in \mathbb R^{n_- \times n_-}$ and $\bV_- \in \mathbb R^{n \times n_-}$ be the diagonal matrix and the orthogonal matrix that are obtained by removing the null part of $\bA_-$ from its diagonalization in such a way that $\bA_- = \bV_- \bL_- \bV_-^\top$ with $\bL_-$ spd. 
\end{definition}
}

It now holds that $\bA = \bA_+ -  \bV_- \bL_- \bV_-^\top $ and the Woodbury matrix identity \cite{woodbury1950inverting} applied to computing the inverse of $\bA$, viewed as a modification of $\bA_+$, gives
\begin{equation}
\label{eq:Wood}
\bA^{-1} = \bA_+^{-1} +  \bA_+^{-1} \bV_- \left(\bL_-^{-1} - \bV_-^\top \bA_+^{-1} \bV_-  \right)^{-1} \bV_-^\top \bA_+^{-1}. 
\end{equation}
This leads to the main theorem in this article in which the new algebraic preconditioner for $\bA$ is defined and the corresponding spectral bound is proved.
\begin{theorem}
For $\tau >1$, let the new preconditioner be defined as 
\[
\bH(\tau) := \bH_+(\tau) + \bA_+^{-1} \bV_- \left(\bL_-^{-1} - \bV_-^\top \bA_+^{-1} \bV_-  \right)^{-1} \bV_-^\top \bA_+^{-1}. 
\]
The eigenvalues of the preconditioned operator satisfy 
\begin{equation}
\label{eq:lambdaHA}
\lambda(\bH(\tau) \bA ) \in \left[\left((1 + 2 \mathcal N_+)\tau  \right)^{-1}  , \mathcal N_+ + 1\right],
\end{equation}
where, once more $ \mathcal N_+$ is the coloring constant with respect to the operator $\bA_+$.
\end{theorem}
\begin{proof}
The estimate for the eigenvalues of $\bH_+(\tau) \bA_+$ in \eqref{eq:lambdaH+A+} is equivalent to 
\[
\left((1 + 2 \mathcal N_+)\tau  \right)^{-1}   \langle \bx, \bA_+^{-1} \bx \rangle \leq \langle \bx, \bH_+(\tau) \bx \rangle \leq (\mathcal N_+ + 1) \langle \bx, \bA_+^{-1} \bx \rangle, \,\forall \bx \in \mathbb R^n .
\]
Adding, $ \langle \bx, \bA_+^{-1} \bV_- \left(\bL_-^{-1} - \bV_-^\top \bA_+^{-1} \bV_-  \right)^{-1} \bV_-^\top \bA_+^{-1} \bx \rangle$ to each term, it holds that
\[
\left((1 + 2 \mathcal N_+)\tau  \right)^{-1}  \langle \bx, \bA^{-1} \bx \rangle \leq \langle \bx, \bH(\tau) \bx \rangle \leq (\mathcal N_+ + 1) \langle \bx, \bA^{-1} \bx\rangle , \,\forall \bx \in \mathbb R^n ,
\]
where \alert{\eqref{eq:Wood}} was applied as well as $\mathcal N_+ \geq 1$ and  $\tau \geq 1$. This is equivalent to \eqref{eq:lambdaHA}.
\end{proof}

\begin{remark}[Cost of the new preconditioner]
\alert{
In order to apply the preconditioner, the matrix $\bA_+^{-1} \bV_-$ must be formed. This can be done by solving iteratively $n_-$ linear systems preconditioned by $\bH_+(\tau)$. It is likely that block Krylov methods would be advantageous. Note that unfortunately $\bA_+^{-1} \bV_-$ is dense as is $\left(\bL_-^{-1} - \bV_-^\top \bA_+^{-1} \bV_-  \right)$. Setting up and applying the second coarse problem $\bA_+^{-1} \bV_- \left(\bL_-^{-1} - \bV_-^\top \bA_+^{-1} \bV_-  \right)^{-1} \bV_-^\top \bA_+^{-1}$ is the most costly part of the algorithm.  

The good news is that the number $n_-$ of columns in $\bV_-$ (which equals the rank of $\bA_-$) satisfies $n_- \leq \sum_{s=1}^N n\s - n$. Consequently, the rank of $\bA_-$ is low compared to the rank $n$ of $\bA$ ($n_- \ll n$) as long as there is little overlap between subdomains. Note that $n_-$ can be (and hopefully is) much smaller even than $\sum_{s=1}^N n\s - n $. 
}
\end{remark}

\section{Numerical Illustration}
\label{sec:Numerical}
The results in this section are obtained using the software FreeFem++ \cite{MR3043640}, GNU Octave \cite{octave} and METIS \cite{METIS}. The linear systems that are considered arise from discretizing with $\mathbb P_1$ finite elements some two-dimensional linear elasticity problems. 

\begin{figure}
\begin{center}
\includegraphics[width=0.3\textwidth]{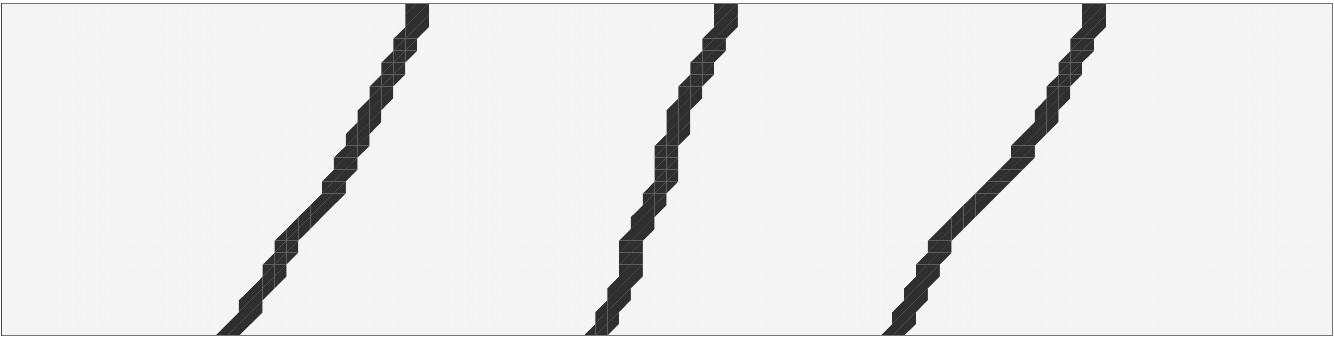}
\includegraphics[width=0.3\textwidth]{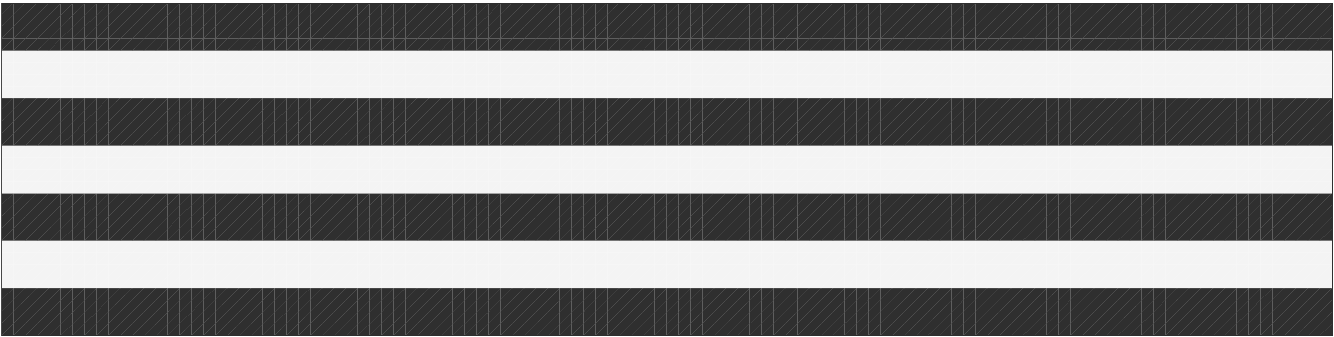}
\end{center}
\caption{Testcase 1 -- \alert{partition ($N=4$) and distribution of $E$ ($10^8$ if white and $10^3$ if dark)}} 
\label{fig:geom1}
\end{figure}

The first test case is posed on the domain $\Omega = [4,1]$ discretized by $112 \times 28$ elements. The problem size is $n = 6496$ degrees of freedom. The coefficients in the linear elasticity equation are $\nu = 0.3$ for Poisson's ratio and 
\[
E(x,y) = 10^8 \text{ if } y \in [1/7,2/7] \cup [3/7,4/7] \cup [5/7, 6/7]; \quad E(x,y) = 10^3 \ \text{ otherwise}.  
\]
The domain is partitioned into 4 subdomains with Metis. No overlap is added. Figure~\ref{fig:geom1} shows both the partition into subdomains and the distribution of $E$. For this problem, the coloring constants with respect to $\bA$ and $\bA_+$ are $\mathcal N = 2$, and $\mathcal N_+ = 3$. The problem is solved with the one-level Additive Schwarz (AS), the two-level AS with the GenEO coarse space from ~\cite{spillane:hal-03186276}[Section 5.2.2] and the new method. The value of the threshold $\tau$ for the last two methods is chosen to be $\tau = 10$. The theoretical bounds for GenEO and the new method is that the eigenvalues are in the interval $[\alert{1/50=0.02} ,3]$ and $[1/70\approx 0.014,4]$, respectively. The $\bA$-norm of the error at each iteration of the preconditioned conjugate gradient is represented in Figure~\ref{fig:cv1}. The quantities of interest are in Table~\ref{tab:testcase1}. The one-level method is not efficient on this problem. This was to be expected. Both the GenEO solver and the new solver converge fast. With $\tau = 10$ in both methods, the coarse space for the new method is larger than with GenEO (58 \textit{versus} 49 coarse vectors). For the new method there is also an additional problem of size 49. 
The results show that the new preconditioner converges a little bit faster than GenEO. A study with more values of all the parameters is needed to compare GenEO and the new solver as the parameter $\tau$ does not play exactly the same role in the setup of both preconditioners. Since there is a lot more information injected into GenEO (through the Neumann matrices), it is expected that GenEO will be more efficient. However the new method has the very significant advantage of being algebraic, and being almost as efficient as GenEO would be an achievement. 
\begin{figure}
\begin{center}
\includegraphics[width=0.8\textwidth]{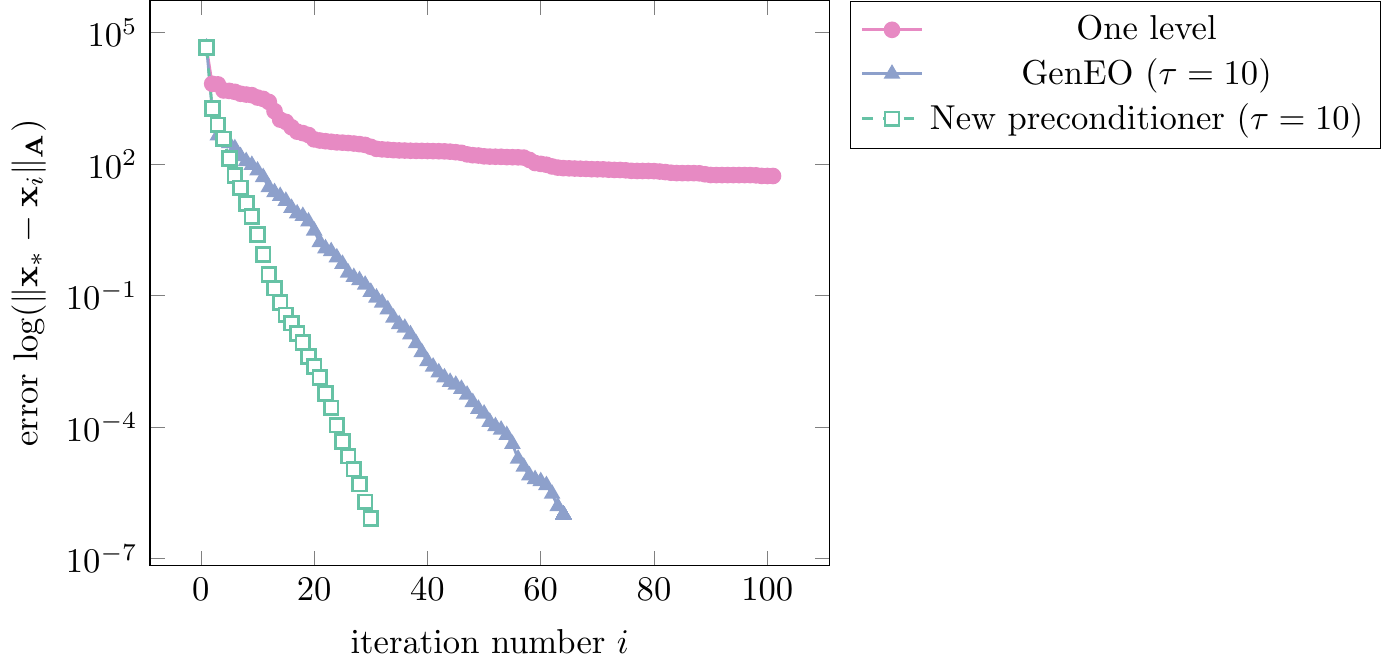}
\end{center}
\caption{Testcase 1 -- Convergence history for the one-level method, the two-level GenEO method and the new method.}
\label{fig:cv1}
\end{figure}

\begin{table}
\begin{center}
\begin{tabular}{c|c|c|c|c|c|c}
  & $\lambda_{\min}$ & $\lambda_{\max}$ & $\kappa$ & It & $\#V^0$ & $n_-$  \\ 
\hline
One-level AS & $2\cdot10^{-4}$ & 2.0 & $1.0\cdot 10^4$ & $>$100 & 0 & 0 \\
Two-level AS with GenEO & 0.059 & 3.0 & 51 & 65 & 49 & 0 \\
New method & 0.24 & 2.93 & 12 & 30 & 58 & 49 \\
\end{tabular}
\end{center}
\caption{\alert{Testcase 1 -- Extreme eigenvalues ($\lambda_{\min}$ and $\lambda_{\max}$), condition number ($\kappa$), iteration count (It), size of coarse space ($\#V^0$), and size of second coarse space in new method ($n_- = \operatorname{rank}(\bA_-)$)}}
\label{tab:testcase1}
\end{table}

It is very good news that the coarse space and the space $V_-$ did not explode on the previous test case. The second test case is a rather easy problem posed on $\Omega = [1,1]$ with a distribution of both coefficients that is homogeneous: $\nu = 0.3$ and $E = 10^8$. Two partitions are considered: one into $N=16$ regular subdomains and the other into $N=16$ subdomains with Metis. No overlap is added to the subdomains. The results are presented in Table~\ref{tab:testcase2}. For the problem with regular subdomains, the new method selects a coarse space of size $44$ (\textit{versus} 40 for GenEO). This means, that even without the knowledge of the Neumann matrix, a coarse space is constructed that has almost the same number of vectors as the optimal coarse space for this problem which consists of $3 \times 12 = 36$ rigid body modes (there are 4 non-floating subdomains). \alert{Of course the second coarse space also adds to the cost.} 
\begin{table}
\begin{center}

\begin{tabular}{c|c|c|c|c|c|c|c|c|c|c|c|c}
 & \multicolumn{6}{c|}{\textbf{$N=16$ regular subdomains}} & \multicolumn{6}{c}{\textbf{$N=16$ subdomains with Metis}} \\ 
  & $\lambda_{\min}$ & $\lambda_{\max}$ & $\kappa$ & It & $\#V^0$ & $n_-$  & $\lambda_{\min}$ & $\lambda_{\max}$ & $\kappa$ & It & $\#V^0$ & $n_-$  \\   
\hline                                                                                                                                                      
One-level AS & $2\cdot10^{-3}$ & 4.0 & 1996 & 97 & 0 & 0     &  $1.7\cdot10^{-3}$ & 3.0 & 1817 & $>$100 & 0 & 0 \\                            
Two-level AS with GenEO & 0.07 & 4.0 & 60 & 61 & 40 & 0      & 0.095 & 3.4 & 36 & 54 & 74 & 0 \\
New method & 0.19 & 4.0 & 21 & 39 & 44 & 24                                  &  0.26 & 3.0 & 11.3 & 31 & 117  & 94                                             
\end{tabular}
\end{center}
\caption{\alert{Testcase 2 -- Extreme eigenvalues ($\lambda_{\min}$ and $\lambda_{\max}$), condition number ($\kappa$), iteration count (It), size of coarse space ($\#V^0$), and size of second coarse space in new method ($n_- = \operatorname{rank}(\bA_-)$)}}
\label{tab:testcase2}
\end{table}

\section{Conclusion}
\label{sec:Conclusion}

\alert{A new algebraic preconditioner was defined for the first time and bounds for the spectrum of the resulting preconditioned operator were proved. They are independent of the number of subdomains and any parameters in the problem. The new preconditioner has two coarse spaces. One of them is dense and a sparse approximation is under investigation. The full length article \cite{AlgebraicGenEO} proposes variants of the new preconditioner that have cheaper choices for $\bH_+$ and less exotic coarse solves.} 

%
%
%
\bibliographystyle{abbrv}
\bibliography{AbstractGenEO}



%
%
%

\end{document}